\documentclass[amsfonts,reqno,11pt]{amsart}

\usepackage{epsfig}
\usepackage{amssymb}
\usepackage{xcolor}
\usepackage{tikz}
\usepackage{float}
\usepackage{graphicx}

\usepackage{float}
\usepackage{graphicx}
\usepackage{epstopdf}
\usepackage{amsmath}

\usepackage[colorlinks=true, allcolors=blue]{hyperref}
\newtheorem{thm}{Theorem}[section]

\newtheorem{lem}{Lemma}[section]

\newtheorem{rem}{Remark}[section]
\newtheorem{prp}{Proposition}[section]

\newtheorem{prob}{Problem}[section]

\newcommand{\R}{\mathbb{R}}

\newcommand{\Rn}{\mathbb{R}^n}

\newcommand{\sphere}{\mathbb{S}^{n-1}}

\def\bpf{\begin{proof}}
	\def\epf{\end{proof}}

\title[A solution to Banach conjecture]{A solution to Banach conjecture}
\author{Ning Zhang}

\address{School of Mathematics and Statistics, Huazhong University of Science and Technology, 1037 Luoyu Road, Wuhan, Hubei 430074, China}
\email{nzhang2@hust.edu.cn}



\begin{document}
\maketitle

\begin{abstract}
In this paper, we begin by constructing global linear maps on \((n-2)\)-dimensional subspaces, derived from the local continuity of linear transformations among central sections of a convex body. Using these linear maps, we subsequently establish a full proof of Banach’s isometric subspace problem in finite-dimensional spaces, extending Gromov’s earlier results.
\end{abstract}

\section{Introduction}
In this paper, we establish a positive answer to the following problem for finite $n$.
\begin{prob}
    Let $(V,\|\cdot\|)$ be a normed vector space (over $\R$) such that for some fixed $n$, $2\leq n\leq \mbox{dim}(V)$, all $n$-dimensional linear subspaces of $V$ are isometric. Is $\|\cdot\|$ neccessarily a Eulidean norm (i.e. an inner product one)?
\end{prob}

This problem goes back to Banach in 1932 \cite[Remarks on Chapter XII]{Banach1932} and is often referred to as the "Banach conjecture." The conjecture asserts that the answer is always affirmative. It has been confirmed in many, though not all, dimensions. Auerbach, Mazur, and Ulam established the conjecture for the case $n=2$ in \cite{AMU}. Dvoretzky \cite{Dvoretzky1959} later proved it for infinite-dimensional spaces $V$ and all $n\geq 2$. Gromov \cite{Gromov1967} showed that the answer is positive whenever $n$ is even or $\dim(V)\geq n+2$. Bor, Hern\'andez-Lamoneda, Jim\'enez-Desantiago, and Montejano extended Gromov’s theorem to all $n$ congruent to $1$ modulo $4$, with the exception of $n=133$; see \cite{BHJM2021}. More recently, Ivanov, Mamaev, and Nordskova \cite{IMN2023} obtained an affirmative solution for $n=3$. This outcome leads from the analysis of sections back to the entire convex body, which provides the original insight for addressing this problem. The problem can also be formulated for complex normed spaces; see \cite{BM2021}, \cite{Gromov1967}, and \cite{M1971} for further details.

Dvoretzky and Gromov reduced the problem to the situation where $\mbox{dim}(V)=n+1$ for some finite integer $n\geq 2$. Then, the problem can be restated in geometric terms as follows.

\begin{prob}
    Let $n\geq 3$ and let $K\subseteq \Rn$ be an origin-symmetric convex body. Assume that all intersections of $K$ with $(n-1)$-dimensional linear subspaces are linearly equivalent. Is $K$ neccessarily an ellipsoid?
\end{prob}

The symmetry condition follows directly from Montejano in \cite{Montejano1991}: If $K\subseteq \Rn$ is a convex body and all intersections of $K$ with $n$-dimensional linear subspaces are affine equivalent, then either $K$ is symmetric with respect to $0$, or $K$ is a (not necessarily centered) ellipsoid. Our principal result is the following theorem, which provides affirmative answers for every finite $n \geq 4$.

\begin{thm}\label{thm 1.1}
    Let $K \subset \Rn$ with $n \ge 4$ be an origin-symmetric convex body. Fix some direction $\xi_0 \in \sphere$. Suppose that for every $\xi \in \sphere$, there exists a linear map $\phi_\xi \in GL(n)$ such that
\[
K \cap \xi^\perp = \phi_\xi\bigl(K \cap \xi_0^\perp\bigr)
\]
and moreover $\phi_\xi(\xi_0) = \xi$. Then $K$ must be an ellipsoid.
\end{thm}

Indeed, Banach's conjecture requires only a linear map on $G(n,n-1)$. However, to complete the proof, we extend this linear map to the entire space by setting $\phi_\xi(\xi_0) = \xi$.

The proof of Theorem \ref{thm 1.1} is divided into four steps. We first use the level set method (see S\"uss's conjecture \cite{LZh2024, Ryabogin2012, Zh2018} for more details) to show that any point $\theta$ on $\sphere\cap \xi_0^\perp$ can form a subset $\Lambda[\theta]$ of $\sphere$ with a nonempty interior due to the local continuity of $\phi_\xi$ with respect to $\xi$. Next, we linearly transform some small spherical ball on $\sphere\cap \xi_0^\perp\cap \theta^\perp$ into the $\sphere\cap \xi^\perp\cap \eta^\perp$ where $\eta\in (\Lambda[\theta])^o$ and $\phi_\xi(\theta)=\eta$. Then, we use the local continuity of $\phi_\xi$ with respect to $\xi$ to generate global linear transformations of the $(n-2)$-dimensional subspace on $\sphere\cap\xi_0^\perp$. Once these linear transformations are obtained, Gromov's result allows us to conclude the theorem.

\section{Notation and Preliminaries}\label{s.2}
This section mainly presents some basic content, such as symbols and definitions. For more detailed content, readers can refer to the books by  Fuente \cite{Fu}, Gardner \cite{Ga}, and Schneider \cite{Sch}.

Let $\mathbb{R}^n$ denote the standard $n$-dimensional Euclidean space. For $x \in \mathbb{R}^n$, we write $\|x\|$ for its Euclidean norm. The {\it unit ball} is defined as $B^n = \{x \in \mathbb{R}^n : \|x\| \le 1\}$, and the {\it unit sphere} as $S^{n-1} = \{x \in \mathbb{R}^n : \|x\| = 1\}$. We use $\{e_1, e_2, \dots, e_n\}$ to denote the standard basis of $\mathbb{R}^n$. For any $\xi \in S^{n-1}$, we define
\[
\xi^\perp := \{x \in \mathbb{R}^n : \langle x, \xi \rangle = 0\}
\]
to be the hyperplane orthogonal to $\xi$, where $\langle x, \xi \rangle$ is the usual inner product on $\mathbb{R}^n$. Furthermore, we set
\[
\xi_+^\perp := \{x \in \mathbb{R}^n : \langle x, \xi \rangle \ge 0\}
\quad\text{and}\quad
\xi_-^\perp := \{x \in \mathbb{R}^n : \langle x, \xi \rangle \le 0\}
\]
for the two half-spaces determined by the hyperplane $\xi^\perp$.

The {\it Grassmann manifold} of $k$-dimensional subspaces of $\R^n$ is written as $G(n,k)$. Let $M_n(\mathbb{R})$ denote the set of all $n\times n$ matrices with real entries, and let $\det : M_n(\mathbb{R}) \to \mathbb{R}$ be the determinant map. We introduce the notation
\[
GL(n) := \{A \in M_n(\mathbb{R}) : \det A \neq 0\}
\]
for the group of invertible $n\times n$ real matrices,
\[
O(n) := \{A \in GL(n) : A^T A = I\}
\]
for the real orthogonal group, where $A^T$ denotes the transpose of $A$ and $I_n$ the $n\times n$ identity matrix, and
\[
SO(n) := \{A \in O(n) : \det A = 1\}
\]
for the special orthogonal group.
For any $\phi\in GL(n)$, its {\it operator norm} is defined by $\|\phi\|_{op}=\max_{x\in B^n}\|\phi(x)\|$.

A subset of $\mathbb{R}^n$ is said to be {\it convex} if, for any two points in the set, the entire closed line segment joining them lies in the set. A convex set is called a {\it convex body} if it is compact and has non-empty interior. A convex body is {\it strictly convex} if no line segment is contained in its boundary.

A compact set $L$ is called a {\it star body} if the origin $O$ lies in the interior of $L$, each line passing through $O$ intersects $L$ in a (possibly degenerate) line segment, and its {\it Minkowski gauge}, given by
$$\|x\|_L = \inf\{a\ge 0: x \in aL\},$$
is a continuous function on $\mathbb R^n$. The {\it radial function} of $L$ is defined by $\rho_L (x) = \|x\|_L^{-1}$ for all $x\in \mathbb{R}^n\setminus \{O\}$. When $x$ lies on the unit sphere $S^{n-1}$, the value $\rho_L(x)$ represents the distance from the origin to the boundary of $L$ in the direction of $x$.


The {\it level set method} (see \cite{LZh2024, Ryabogin2012, Zh2018} for more details) consists in identifying a function on the sphere that remains unchanged under linear (congruent) transformations relating sections or projections, so that one of its level sets can coincide with the entire sphere.


On the sphere $\sphere$, an {\it equator} is a $(n-2)$-dimensional subsphere of the form $\sphere \cap \xi^\perp$ for some $\xi \in \sphere$. The {\it geodesic distance} between points $\theta_1$ and $\theta_2$, written $d_{\sphere}(\theta_1,\theta_2)$, is defined as the length of the shorter geodesic segment connecting them, which is equal to the angle between $\theta_1$ and $\theta_2$. The spherical ball of radius $r$ centered at $o$ is given by
$$
\{\theta \in \sphere : d_{\sphere}(\theta,o) \leq r\}.
$$

\section{Proof of the main theorem}
We begin the proof with a special case.
\begin{thm}\label{thm 2.1}
    Let $K\subset \Rn$ for $n\geq 3$ be an origin-symmetric star body. If for any $\xi\in\sphere$, $K\cap \xi^\perp$ is a centered ellipsoid, then $K$ is an ellipsoid.
\end{thm}

Theorem \ref{thm 2.1} follows from the classical characterization by Busemann of ellipsoids (see \cite[Theorem 3.1]{Soltan2019}). The convexity of the body $K$ results from the fact that if we have two points in the body, we can examine the plane they span (including the origin). Since the planar section is an ellipsoid by assumption, the entire line segment between $x$ and $y$ is contained in $K$.

Now we can prove the main theorem using the local continuity of $\phi_\xi$. First, we should introduce a couple of Lemmas.

First, we define a set-valued function $\lambda:\sphere\to GL(n)$ to be
$$
\lambda(\xi):=\{\phi_{\xi}\in GL(n): K\cap \xi^\perp=\phi_{\xi}(K\cap\xi_0^\perp)\mbox{ and }\phi_\xi(\xi_0)=\xi\}.
$$
And we define the linear symmetry group $G_{K\cap\xi_0^\perp}$ to be
$$
G_{K\cap\xi_0^\perp}:=\{g\in GL(n): g(K\cap\xi_0^\perp)=K\cap \xi_0^\perp \mbox{ and } g(\xi_0)=\xi_0\}.
$$
Note that for any $\phi_\xi, \psi_\xi\in \lambda(\xi)$, we have 
$$
\psi_\xi^{-1}\phi_\xi(K\cap\xi_0^\perp)=K\cap\xi_0^\perp,
$$
which gives $\phi_\xi\in \psi_\xi \circ G_{K\cap\xi_0^\perp}$; that is
$$
\lambda(\xi)=\phi_\xi\circ G_{K\cap\xi_0^\perp},
$$
for some $\phi_\xi\in \lambda(\xi)$. And thus, the right action of $G_{K\cap\xi_0^\perp}$ on $GL(n)$ gives the quotient $GL(n)/G_{K\cap\xi_0^\perp}$, where the set-valued function $\lambda(x)$ becomes a map $\Phi: \sphere \to GL(n)\backslash G_{K\cap\xi_0^\perp}$ with
$$
K\cap \xi^\perp=\Phi(\xi)(K\cap\xi_0^\perp)
$$
Then we need to show $\Phi$ is continuous.
\begin{lem}\label{lem 3.2}
     Let $K, \phi_\xi$ be as in Theorem \ref{thm 1.1}. Then $\Phi$ is continuous..
\end{lem}
\begin{proof}
Note that the operator norm of $\phi_\xi\in \Phi(\xi)$ is bounded, since 
$$
\|\phi_\xi\|_{op}\leq \frac{\max_{\theta\in \sphere}\rho_K(\theta)}{\min_{\theta\in \sphere}\rho_K(\theta)}.
$$
We may choose an arbitrary sequence $\{\xi_n\}_{n=1}^\infty$ converging to $\xi$, i.e., with $\lim_{n\to\infty}\xi_n=\xi$, and for each $n$ select $\phi_{\xi_n}\in \Phi(\xi_n)$. Then there exists a subsequence $\{\xi_{n_k}\}_{k=1}^\infty$ such that the corresponding operators $\phi_{\xi_{n_k}}$ converge in operator norm to some $\phi\in GL(n)$. And thus, we have
\begin{align*}
    & d_{H}(\phi(K\cap \xi_0),K\cap\xi^\perp)\\
     \leq & d_{H}(\phi(K\cap \xi_0), K\cap\xi_{n_k}^\perp)+d_H(K\cap\xi_{n_k}^\perp, K\cap\xi^\perp)\\
    \leq & d_{H}(\phi(K\cap \xi_0), \phi_{\xi_{n_k}}(K\cap\xi_0^\perp))+d_H(K\cap\xi_{n_k}^\perp, K\cap\xi^\perp)\\
    \leq & c \left|\|\phi\|_{op}-\|\phi_{\xi_{n_k}}\|_{op}\right|+d_H(K\cap\xi_{n_k}^\perp, K\cap\xi^\perp).
\end{align*}
As $k\to \infty$, we have
$$
d_{H}(\phi(K\cap \xi_0),K\cap\xi^\perp)=0.
$$
This implies that $\phi(K\cap \xi_0)=K\cap\xi^\perp$. Moreover,
$$
\phi(\xi_0)=\lim_{k\to\infty}\phi_{\xi_{n_k}}(\xi_0)=\lim_{k\to\infty}\xi_{n_k}=\xi.
$$
and hence $\phi\in \Phi(\xi)$. Consequently, we have $\Phi(\xi_{n_k})\to \Phi(\xi)$ as $k\to \infty$.

This means that for any subsequence $\{\Phi(\xi_{n_k})\}_{k=1}^\infty$, there exists a sub-subsequence $\{\Phi(\xi_{n_{k_l}})\}_{l=1}^\infty$ that converges to $\Phi(\xi)$. Consequently, we conclude that $\Phi(\xi_n) \to \Phi(\xi)$.
\end{proof}

\begin{rem}
   A proof that $\Phi:\sphere\to GL(n)/G_{K\cap\xi_0^\perp}$ is continuous is also provided in \cite[Lemma 1.5]{BHJM2021}.
\end{rem}

For a fixed $\theta \in \sphere \cap \xi^\perp_0$, Our next objective is to introduce the set $\tilde{\Lambda}[\theta]\subseteq \sphere$ by
$$
\tilde{\Lambda}[\theta] := \left\{\,\frac{\phi_\xi(\theta)}{\|\phi_\xi(\theta)\|} : \phi_\xi \in \Phi(\xi)\ \text{and}\ \xi \in \sphere \,\right\}.
$$

In general, $\tilde{\Lambda}[\theta]$ is generated from discrete subsets
$$
\{g(\theta):g\in G_{K\cap\xi_0^\perp}\},
$$
so we need to choose a specific path-connected subset of $\tilde{\Lambda}[\theta]$. Note that the linear symmetry group is compact (\cite[Lemma 2.1]{BHJM2021}), and hence we define $H\subseteq G_{K\cap\xi_0^\perp}$ to be the path-connected component that contains the identity element.
By the continuity of $\Phi$, for each $\eta\in\tilde{\Lambda}[\theta]$ with $\phi_\xi(\theta)=\eta$, we can choose a local continuous map $\phi_\vartheta:U_\xi\to GL(n)$ (see Figure \ref{f.1}), where $U_\xi$ is a neighborhood of $\xi$.
\begin{figure}[h]
	\center	\includegraphics[width=0.8\linewidth]{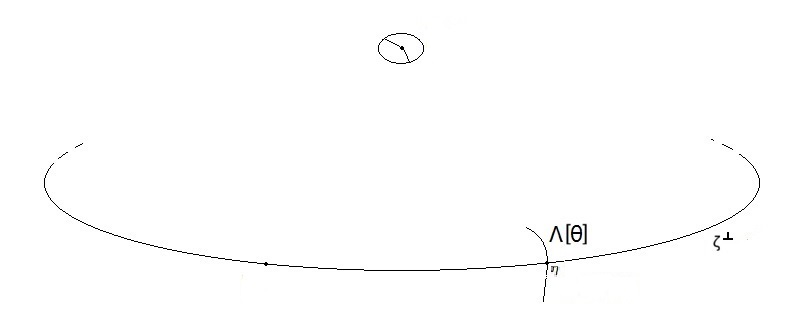}
	\caption{The set $\Lambda_t[\theta]$ passes through $\zeta^\perp$.}
	\label{f.1}
\end{figure}

Fixing $\phi_{\xi_0}$ to be the identity, we can construct a continuous path $\gamma$ from $\theta$ to $\eta$ by piecing together these local continuous maps. This path yields
$$
\phi^{-1}_\xi(\eta)\subseteq \{g(\theta):g\in H\}.
$$
We then define a subset of $\tilde{\Lambda}[\theta]$ by
\begin{align*}
\Lambda[\theta]:=\{&\frac{\phi_\xi\circ g(\theta)}{\|\phi_\xi\circ g(\theta)\|} : \phi_\xi \circ g_\xi\in \Phi(\xi) \mbox{ is local continuous},\ g\in H, \text{and}\ \xi \in \sphere \}.
\end{align*}


We now present one basic proposition of $\Lambda[\theta]$.
\begin{prp}\label{prp 3.1}
 Let $K, \phi_\xi$ be as in Theorem \ref{thm 1.1}. Then the complement $\sphere \setminus \Lambda[\theta]$ consisting of finite connected components, where each component is contained in a spherical ball with radius $r<\frac{\pi}{2}$.
\end{prp}
\begin{proof}
   If $\Lambda[\theta]\neq \sphere$, we need to show that every connected component of $\sphere\backslash \mathrm{Im}(\Lambda_t[\theta])$ is contained in some spherical ball centered at $o$ with radius $r$, namely
   $$
   \{y \in \sphere : d_{\sphere}(y,o)\leq r\}.
   $$
   Observe that for any $\eta\in \Lambda[\theta]$ and any $\zeta\in \sphere\cap \eta^\perp$, the set
   $$
   \{\phi_\vartheta\circ g_\vartheta(\theta)): \vartheta\in \sphere \cap \mathrm{span}(\eta,\zeta)\}\subseteq \Lambda[\theta]
   $$
   traces out a curve that passes through $\eta$ and enters both $\sphere \cap \zeta^\perp_+$ and $\sphere \cap \zeta^\perp_-$ (see Figure \ref{f.1}).
    Consequently, every boundary point $x$ of one branch $C$ of $\sphere \setminus \Lambda[\theta]$ must lie on a spherical ball that is tangent at $x$ and has radius $r_x < \frac{\pi}{2}$. This subsphere can be written as
   $$
   \{y \in \sphere : d_{\sphere}(y, o_x) \leq r_x\}.
   $$
   Otherwise, we may prolong a curve contained in $\Lambda[\theta]$ until it intersects $C$. By the compactness of $\Bar{C}$, this yields a uniform radius $r_0<\frac{\pi}{2}$ such that $C$ lies entirely in a spherical ball of radius $r_0$.

   On the other hand, if $\sphere \setminus \Lambda[\theta]$ has infinitely many connected components, then we can choose a sequence $\{\theta_i\}_{i=1}^\infty$ such that each $\theta_i$ lies in a different connected component. By the compactness of $\sphere\cap\xi_0^\perp$, this sequence admits a convergent subsequence, which contradicts the fact that each connected component is open.
\end{proof}

Our next lemma will provide a local linear transformation among $(n-2)$-dimensional subspaces. 
\begin{lem}\label{lem 3.2}
     Let $K$ and $\phi_\xi$ be as in Theorem \ref{thm 1.1}. Assume that for some $\zeta_0\in\sphere$ and $\theta_0\in \sphere\cap \xi_0^\perp$, there exists a neighborhood $U_{\zeta_0}$ of $\zeta_0$ that is contained in $\Lambda[\theta_0]$. Then, for any $\beta_0\in \sphere\cap \zeta_0^\perp\cap \eta_0^\perp$, one can choose $\vartheta_0\in\sphere\cap \xi_0^\perp\cap\theta_0^\perp$ such that $\vartheta_0\perp \phi^{-1}_{\eta_0}(\eta_0^\perp\cap \beta_0^\perp)$. Consequently, there exists an $(n-2)$-dimensional open spherical ball $B_\delta$ centered at $\vartheta_0$ with the property that for every $\vartheta\in B_\delta$, the subspace $\xi_0^\perp\cap\vartheta^\perp$ is linearly equivalent, via an $SO(2)$-transformation, to some $\xi_0^\perp\cap\varrho^\perp$ for a suitable $\varrho\in \sphere\cap\xi_0^\perp\cap\theta_0^\perp$.
\end{lem}
\begin{proof}
 Before begining the proof, we introduce a new notation: we write $[\zeta,\eta]$ to mean that $\phi_\eta(\theta_0)=\zeta$. For the specific pair $[\zeta_0,\eta_0]$, since $U_{\zeta_0} \subseteq \Lambda[\theta_0]$, there exists an open set $\mathcal{O}_{\eta_0}$ containing $\eta_0$. Furthermore, for each $\zeta \in U_{\zeta_0}$, there is some $\eta \in \mathcal{O}_{\eta_0}$ such that $[\zeta,\eta]$ holds, and distinct points $\zeta_1 \neq \zeta_2$ correspond to the pairs $[\zeta_1,\eta_1]$ and $[\zeta_2,\eta_2]$, gives distinct points $\eta_1 \neq \eta_2$.

    Choose a neighborhood $\mathcal{V}_{\beta_0}$ of $\beta_0$ such that, for every $\beta\in\mathcal{V}_{\beta_0}$ and every $\zeta\in U_{\zeta_0}\cap \beta^\perp$, the preimage $\phi_\eta^{-1}(\sphere\cap \eta^\perp\cap \beta^\perp)$ is well-defined; this is an $(n-2)$-dimensional subspace of $\xi_0^\perp$ containing $\theta_0$. Next, pick $\vartheta_0\in\sphere\cap\xi_0^\perp$ orthogonal to $\phi_{\eta_0}^{-1}(\sphere\cap \eta_0^\perp\cap \xi_0^\perp)$. Then
$$
\bigcup_{\substack{\eta\in\mathcal{O}_{\eta_0} \\ \beta\in\mathcal{V}_{\varsigma_0}}}\{\vartheta\in \sphere\cap \xi_0^\perp: \vartheta\perp\phi_\eta^{-1}(\sphere\cap \eta^\perp\cap \varsigma^\perp)\}
$$
forms an $(n-3)$-dimensional open neighborhood of $\vartheta_0$ in $\sphere\cap\xi_0^\perp\cap\theta_0^\perp$, which we denote by $V_{\vartheta_0}$.

    Now choose a small $(n-2)$-dimensional open spherical ball $B_{\delta}\subseteq \sphere\cap\xi_0^\perp$ centered at $\vartheta_0$ such that $B_{\delta}\cap \theta_0^\perp \subseteq V_{\vartheta_0}$. For any $\eta\in \mathcal{O}_{\zeta_0}$ and $\varsigma\in\mathcal{V}_{\varsigma_0}$, let $A_{\eta}$ denote the matrix associated with $\phi_\eta$. Then $A_\eta^{-T}(B_\delta)$ represents the collection of orthogonal sets $\cup_{\vartheta\in B_{\delta}}\phi_\eta(\vartheta^\perp)$ because
$$
\langle A_\eta \theta, A_\eta^{-T}\vartheta\rangle= \langle \theta, \vartheta \rangle,
$$
where $A_\eta^{-T}=(A_\eta^{-1})^T$ is the transpose of the inverse of $A_\eta$. Furthermore, define
$$
\mathcal{A}_\eta:=\bigcup_{\vartheta\in B_{\delta}}\frac{A_\eta^{-T}(\vartheta)}{\|A_\eta^{-T}(\vartheta)\|}.
$$

    We now have $\varsigma_0=\frac{A_{\eta_0}^{-T}(\vartheta_0)}{\|A_{\eta_0}^{-T}(\vartheta_0)\|}$ and consider the Gnomonic projection at $\varsigma_0$, denoted by 
    $$
    P_{\varsigma_0}: \sphere\cap (\varsigma_0^\perp)_+\to H_{\varsigma_0}.
    $$
    It is clear that $P_{\varsigma_0}(\mathcal{A}_\eta)$ is a $(n-2)$-dimensional ellipsoid. Moreover, we set
    $$
    P_{\vartheta_0}: \sphere\cap (\vartheta_0^\perp)_+\to H_{\vartheta_0}.
    $$
    to be the Gnomonic projection at $\vartheta_0$. Then we have the linear maps $\psi_{\eta}:P_{\vartheta_0}(B_\delta)\to P_{\varsigma_0}(\mathcal{A}_\eta)$ to be
    \[
    \psi_{\eta}(P_{\vartheta_0}(\vartheta))=P_{\varsigma_0}\left(\frac{A_\eta^{-T}(\vartheta)}{\|A_\eta^{-T}(\vartheta)\|}\right).
    \]

    Observe that $\mathcal{A}_\eta$ lies in $\sphere \cap \eta^\perp$. Consequently, any geodesic circle on $\sphere$ that is orthogonal to $\eta^\perp$ determines a line $l$ that is orthogonal to $P_{\varsigma_0}(\mathcal{A}_\eta)$.
    
    Given two parallel ellipsoids $P_{\varsigma_0}(\mathcal{A}_{\eta_1})$ and $P_{\varsigma_0}(\mathcal{A}_{\eta_2})$ and any line $l$ orthogonal to both, we look at the intersections $l\cap P_{\varsigma_0}(\mathcal{A}_{\eta_1})$ and $l\cap P_{\varsigma_0}(\mathcal{A}_{\eta_2})$, which we denote by $P_{\varsigma_0}(\varsigma_1)$ and $P_{\varsigma_0}(\varsigma_2)$, respectively. Since $P_{\varsigma_0}^{-1}l$ is a geodesic half-circle, it follows directly that $\varsigma_1^\perp\cap \eta_1^\perp$ and $\varsigma_2^\perp\cap \eta_2^\perp$ coincide.

    Now we examine the preimages $\frac{A_{\eta_1}^T\varsigma_1}{\|A_{\eta_1}^T\varsigma_1\|}$ and $\frac{A_{\eta_2}^T\varsigma_2}{\|A_{\eta_2}^T\varsigma_2\|}$, which we denote by $\vartheta_1$ and $\vartheta_2$, respectively. It is evident that there exists a linear transformation between $\sphere\cap \xi_0^\perp\cap \vartheta_1^\perp$ and $\sphere\cap \xi_0^\perp\cap \vartheta_2^\perp$.

Therefore, we can introduce a map $P_l:P_{\varsigma_0}(\sphere\cap\eta_1^\perp\cap(\varsigma_0^\perp)_+)\to P_{\varsigma_0}(\sphere\cap\eta_2^\perp\cap(\varsigma_0^\perp)_+)$ defined by
$$
P_l(l\cap P_{\varsigma_0}(\sphere\cap\eta_1^\perp\cap(\varsigma_0^\perp)_+) = l\cap P_{\varsigma_0}(\sphere\cap\eta_2^\perp\cap(\varsigma_0^\perp)_+)
$$
whenever $l\perp P_{\varsigma_0}(\sphere\cap\eta_1^\perp\cap(\varsigma_0^\perp)_+)$.
In this way, we can examine the linear map 
$$
\Psi_{\eta_1,\eta_2}:=\psi^{-1}_{\eta_2}\circ P_l\circ \psi_{\eta_1}
$$
with $l$ passing through $ P_{\varsigma_0}(\mathcal{A}_{\eta_1})$ and $ P_{\varsigma_0}(\mathcal{A}_{\eta_2})$ and there exists a linear transformation between $\sphere\cap \xi_0^\perp\cap (\Psi_{\eta_1,\eta_2})(\vartheta)^\perp$ and $\sphere\cap \xi_0^\perp\cap \vartheta^\perp$.

    Note that, for any line $l$ in $H_{\varsigma_0}$, the preimage of $l$ by $P_{\varsigma_0}$ is a geodesic half circle. For the parallel ellipsoids from $\cup_{\eta\in \mathcal{O}_{\eta_0}}P_{\varsigma_0}(\mathcal{A}_\eta)$, we claim that one of the following conditions holds.
    \begin{enumerate}
        \item The parallel ellipsoids can form an elliptical cylinder in $H_{\varsigma_0}$.
        \item There exists a $SO(2)$-affine invariant in $\bigcup_{\vartheta\in B_\delta}\vartheta^\perp\cap\xi_0^\perp$.
        \item For any $\vartheta_1, \vartheta_2\in B_\delta$, there exists a linear transformation between $\xi_0^\perp\cap\vartheta_1^\perp$ and $\xi_0^\perp\cap\vartheta_2^\perp$.
    \end{enumerate}
   To verify the statement, we discuss in cases.

   Case 1. Suppose that $\Psi_{\eta_1,\eta_2}(P_{\vartheta_0}(\vartheta))=P_{\vartheta_0}(\vartheta)$ holds for every $\vartheta\in B_\delta$ and for all $\eta_1,\eta_2\in \mathcal{O}_{\eta_0}$ such that the positions $P_{\varsigma_0}(\mathcal{A}_{\eta_1})$ and $P_{\varsigma_0}(\mathcal{A}_{\eta_2})$ are parallel. In this situation, the corresponding parallel ellipsoids can jointly form an elliptical cylinder in $H_{\varsigma_0}$ (see Figure \ref{f.2}).
\begin{figure}[h]
	\center	\includegraphics[width=0.8\linewidth]{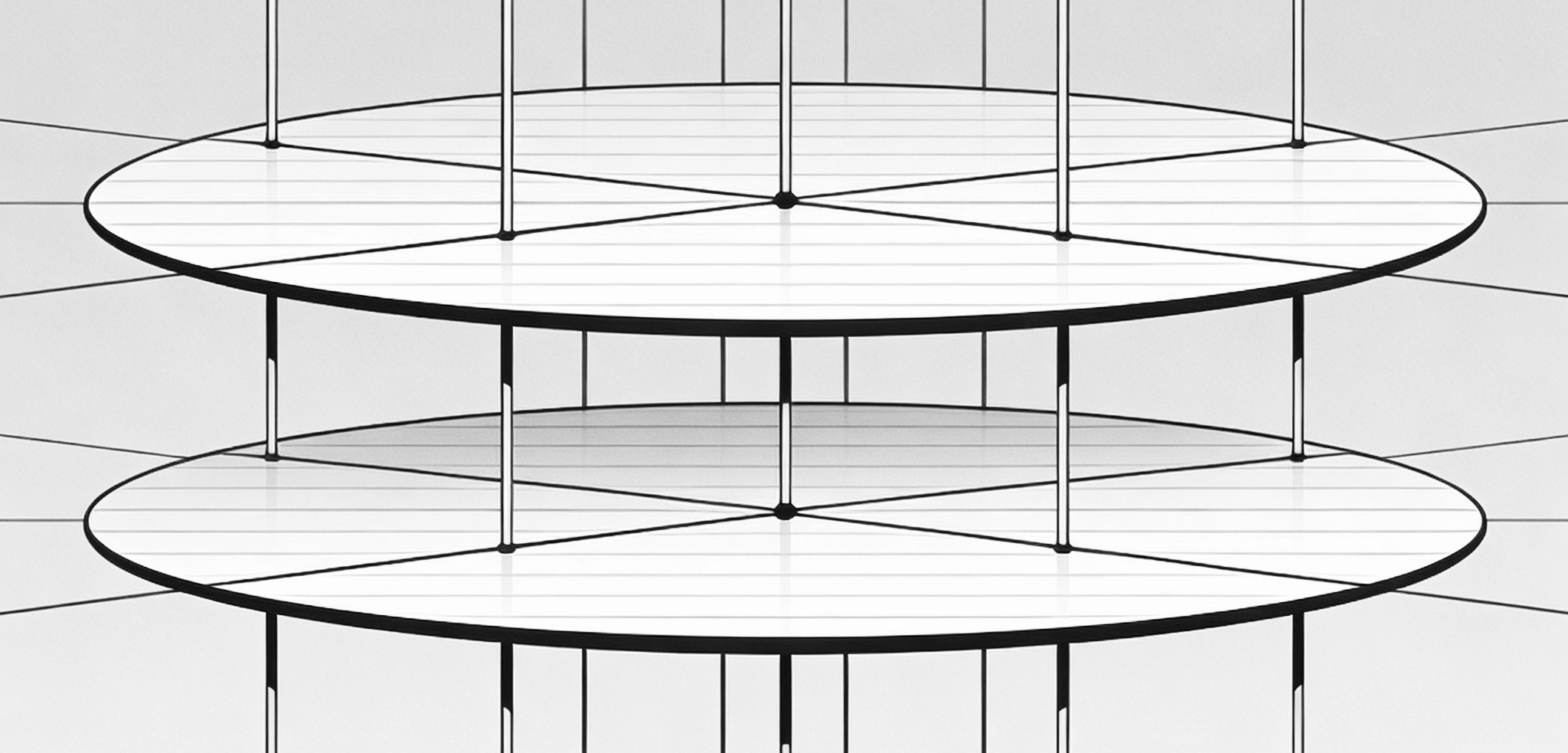}
	\caption{The parallel ellipsoids.}
	\label{f.2}
\end{figure}

   Case 2. Suppose that there exist $\eta_1,\eta_2\in\mathcal{O}_{\eta_0}$ such that
$$
\Psi_{\eta_1,\eta_2}(P_{\vartheta_0}(\vartheta_0))=P_{\vartheta_0}(\vartheta_0)\quad\text{and}\quad
\Psi_{\eta_1,\eta_2}(P_{\vartheta_0}(\vartheta_1))=P_{\vartheta_0}(\vartheta_2)
$$
for two distinct points $\vartheta_1,\vartheta_2\in B_\delta$. If $\|\Psi_{\eta_1,\eta_2}(P_{\vartheta_0}(\vartheta_1))\|_{H_{\vartheta_0}}<\|P_{\vartheta_0}(\vartheta_2)\|_{H_{\vartheta_0}}$ (otherwise replace $\Psi_{\eta_1,\eta_2}$ by its inverse), then the sequence $\{\Psi^n_{\eta_1,\eta_2}(P_{\vartheta_0}(\vartheta_1))\}_n$ converges to $\vartheta_0$. Furthermore,
$$
\{\Psi^n_{\eta_1,\eta_2}(st\,P_{\vartheta_0}(\vartheta_1)+s(1-t)\,P_{\vartheta_0}(\vartheta_2)) : 0\le t\le 1,\ s\in\mathbb{R}\}_n
$$
fills out every $\vartheta\in B_\delta$. Hence, for each such $\vartheta$, there exists a linear isomorphism between $\xi_0^\perp\cap\vartheta^\perp$ and $\xi_0^\perp\cap\vartheta_0^\perp$ (Condition 3).

Case 3. Suppose that there exist $\eta_1,\eta_2\in\mathcal{O}_{\eta_0}$ such that
$$
\Psi_{\eta_1,\eta_2}(P_{\vartheta_0}(\vartheta_0))=P_{\vartheta_0}(\vartheta_0)\quad\text{and}\quad
\Psi_{\eta_1,\eta_2}(P_{\vartheta_0}(\vartheta_1))=P_{\vartheta_0}(\vartheta_2)
$$
for two distinct points $\vartheta_1,\vartheta_2\in B_\delta$. If $\|\Psi_{\eta_1,\eta_2}(P_{\vartheta_0}(\vartheta_1))\|_{H_{\vartheta_0}}=\|P_{\vartheta_0}(\vartheta_2)\|_{H_{\vartheta_0}}$, then the sequence $\{\Psi^n_{\eta_1,\eta_2}(P_{\vartheta_0}(\vartheta_1))\}_n$ either forms a dense set of a circle or finite points in a circle. If it yields finite points in a circle, we can consider the geodesic segment $\gamma$ connecting $\varsigma_1$ and $\varsigma_2$, which are located on $\eta_1^\perp$ and $\eta_2^\perp$ respectively. By the local continuity of $\phi_{\eta}(B_\delta)$ with respect to $\eta$, there will be a continuous path 
    $$
    \bigcup_{\substack{\varsigma\in \gamma\\
    \gamma\perp \mathcal{A}_{\eta}}}\frac{A_{\eta}^T\varsigma}{\|A_{\eta}^T\varsigma\|}
    $$
    joining $\vartheta_1$ and $\vartheta_2$, denoted by $\sigma$. For any $\vartheta\in\sigma$ and corresponding $\eta$, we have
    $$
\Psi_{\eta_1,\eta}(P_{\vartheta_0}(\vartheta_0))=P_{\vartheta_0}(\vartheta_0)\quad\text{and}\quad
\Psi_{\eta_1,\eta}(P_{\vartheta_0}(\vartheta_1))=P_{\vartheta_0}(\vartheta)
$$
and
$$
\|\Psi_{\eta_1,\eta}(P_{\vartheta_0}(\vartheta_1))\|_{H_{\vartheta_0}}=\|P_{\vartheta_0}(\vartheta)\|_{H_{\vartheta_0}}.
$$
Otherwise, we fall into the other cases. By the continuity of $\sigma$, there exists some $\vartheta \in \sigma$ such that the sequence $\{\Psi^n_{\eta_1,\eta}(P_{\vartheta_0}(\vartheta_1))\}_n$ becomes dense in a circle. Furthermore, the family $\{\Psi^n_{\eta_1,\eta}\}_n$ is dense in the set of $SO(2)$-rotations; thus, by the closedness of linear transformations, there exists an $SO(2)$-affine invariant contained in $\bigcup_{\vartheta\in B_\delta}\vartheta^\perp\cap\xi_0^\perp$ (Condition 2).

Case 4. Suppose that there exists $\eta_1,\eta_2\in\mathcal{O}_{\eta_0}$ such that
$$
\Psi_{\eta_1,\eta_2}(P_{\vartheta_0}(\vartheta_0))=P_{\vartheta_0}(\vartheta_1)\quad\text{and}\quad
\Psi_{\eta_1,\eta_2}(P_{\vartheta_0}(\vartheta_2))=P_{\vartheta_0}(\vartheta_3)
$$
for two distinct points $\vartheta_1,\vartheta_2,\vartheta_3\in B_\delta$.
Next, we decompose $\Psi_{\eta_1,\eta_2}$ into two maps, $\Psi^1_{\eta_1,\eta_2}$ and $\Psi^2_{\eta_1,\eta_2}$, defined by
$$
\Psi^1_{\eta_1,\eta_2}(P_{\vartheta_0}(\vartheta)) = P_{\vartheta_0}(\vartheta) + P_{\vartheta_0}(\vartheta_1) - P_{\vartheta_0}(\vartheta_0)
$$
and
$$
\Psi^2_{\eta_1,\eta_2}(P_{\vartheta_0}(\vartheta)) = \Psi_{\eta_1,\eta_2}(P_{\vartheta_0}(\vartheta)) - P_{\vartheta_0}(\vartheta_1) + P_{\vartheta_0}(\vartheta_0).
$$
Clearly,
$$
\Psi_{\eta_1,\eta_2} = \Psi^2_{\eta_1,\eta_2} \circ \Psi^1_{\eta_1,\eta_2}.
$$
If $\Psi^2_{\eta_1,\eta_2}$ is the identity map, then $\Psi_{\eta_1,\eta_2}$ is a translation, which produces an $SO(2)$-affine invariant on $\bigcup_{\vartheta\in B_\delta}\vartheta^\perp\cap\xi_0^\perp$ (Condition 2).

If $\Psi^2_{\eta_1,\eta_2}$ is not the identity, we have
$$
\Psi^2_{\eta_1,\eta_2}(P_{\vartheta_0}(\vartheta_0)) = P_{\vartheta_0}(\vartheta_0)
$$
and
$$
\Psi^2_{\eta_1,\eta_2}(P_{\vartheta_0}(\vartheta_2)) 
= P_{\vartheta_0}(\vartheta_3) - P_{\vartheta_0}(\vartheta_1) + P_{\vartheta_0}(\vartheta_0) 
\neq P_{\vartheta_0}(\vartheta_2),
$$
which brings us back to Case 2 or Case 3.

When $\Psi^2_{\eta_1,\eta_2}$ falls into Case 2, we again obtain, for each such $\vartheta$, a linear isomorphism between $\xi_0^\perp\cap\vartheta^\perp$ and $\xi_0^\perp\cap\vartheta_0^\perp$ (Condition 3).

When $\Psi^2_{\eta_1,\eta_2}$ falls into Case 3, the composition $\Psi^2_{\eta_1,\eta_2}\circ \Psi^1_{\eta_1,\eta_2}$ yields an $SO(2)$-affine invariant on $\bigcup_{\vartheta\in B_\delta}\vartheta^\perp\cap\xi_0^\perp$ (Condition 2).

To complete the proof of the Lemma, assume that there exists an open set $U \subset B_\delta$ such that, for every $\vartheta\in U$, the subspace $\xi_0^\perp\cap\vartheta^\perp$ is not linearly equivalent, via any $SO(2)$-transformation, to a subspace of the form $\xi_0^\perp\cap\varrho^\perp$ for some $\varrho\in \sphere\cap\xi_0^\perp\cap\theta_0^\perp$. According to the claim, either the family of parallel ellipsoids forms an elliptic cylinder in $H_{\varsigma_0}$, or the $SO(2)$-affine transformation is tangent to $\sphere\cap\xi_0^\perp\cap\theta_0^\perp$. In either situation, there exists a region contained in all parallel ellipsoids $P_{\varsigma_0}(\mathcal{A}_{\eta})$ that does not intersect $\psi_{\eta}(B_\delta\cap \theta_0^\perp)$.

However, if we look at the set
$$
\bigcup_{\beta\in \mathcal{V}_{\beta_0}\cap \mathrm{span}\{\beta_0,\zeta_0\}} \beta^\perp\cap U_{\zeta_0},
$$
we obtain an open neighbourhood of $\zeta_0$. Consequently, within the ellipsoids $P_{\varsigma_0}(\mathcal{A}_{\eta_0})$ there exists a small neighbourhood $V$ of $P_{\varsigma_0}(\varsigma_0)$ such that, for every $P_{\varsigma_0}(\varsigma)\in V$, there is some $\beta\in \mathcal{V}_{\beta_0}\cap \mathrm{span}\{\beta_0,\zeta_0\}$ and a pair $[\zeta,\eta]$ with $\beta\perp \zeta$ for which the geodesic arc $C$ passing through $\beta$ and $\eta$ also passes through $\eta_0$. Moreover, $P_{\varsigma_0}(C)$ passes through $P_{\varsigma_0}(\varsigma)$. This implies that, for any $P_{\varsigma_0}(\varsigma)\in V$, we have $\psi^{-1}_\eta(P_{\varsigma_0}(\varsigma))\in P_{\vartheta_0}(B_\delta\cap \theta_0^\perp)$. Therefore, by the claim, either the subspace $\xi_0^\perp\cap\vartheta^\perp$ is linearly equivalent, via an $SO(2)$-transformation, to some $\xi_0^\perp\cap\varrho^\perp$ for a suitable $\varrho\in \sphere\cap\xi_0^\perp\cap\theta_0^\perp$, or, for any $\vartheta_1,\vartheta_2\in B_\delta$, there exists a linear transformation between $\xi_0^\perp\cap\vartheta_1^\perp$ and $\xi_0^\perp\cap\vartheta_2^\perp$.  
\end{proof}

Now we can prove the main theorem.
\begin{proof}[Proof of the Theorem \ref{thm 1.1}]
    We only need to show that for any $\vartheta_0\in \sphere\cap\xi_0^\perp$, there exists a neighbourhood $\mathcal{V}_{\vartheta_0}$ of $\vartheta_0$ such that for all $\vartheta\in \mathcal{V}_{\vartheta_0}$, $\xi_0^\perp\cap \vartheta^\perp$ are linearly equivalent to each other.

    Then by the closeness of linearly equivalence and $\sphere\cap \xi_0^\perp$ being connected, we can get $K$ to be the ellipsoid by Gromov's result \cite{Gromov1967}.
    
    To complete the proof, Lemma \ref{lem 3.2} guarantees the existence of an $SO(2)$-linear invariant in $\mathcal{V}_{\vartheta_0}$. Next, suppose there is a neighbourhood $\mathcal{V}_k$ in which an $SO(k)$-linear invariant exists. In that case, we can select an $(n-2)$-dimensional geodesic sphere $\theta^\perp\cap\xi_0^\perp$ that is tangent to the $SO(k)$-orbit at $\vartheta_0$. Applying Lemma \ref{lem 3.2} once more, we then obtain an $SO(k+1)$-linear invariant in a smaller neighbourhood $\mathcal{V}_{k+1}$. Proceeding inductively, we thus arrive at a neighbourhood $\mathcal{V}$ of $\vartheta_0$ such that, for every $\vartheta\in \mathcal{V}_{\vartheta_0}$, the spaces $\xi_0^\perp\cap \vartheta^\perp$ are all linearly equivalent to each other.
\end{proof}


\end{document}